\documentclass[english]{article}
\usepackage[T1]{fontenc}
\usepackage[latin9]{inputenc}
\usepackage{textcomp}
\usepackage{amsthm}
\usepackage{amstext}
\usepackage{amssymb}

\makeatletter

\providecommand{\tabularnewline}{\\}

\newcommand{\lyxaddress}[1]{
\par {\raggedright #1
\vspace{1.4em}
\noindent\par}
}
\theoremstyle{plain}
\newtheorem{thm}{\protect\theoremname}
\theoremstyle{plain}
\newtheorem{lem}[thm]{\protect\lemmaname}
\ifx\proof\undefined
\newenvironment{proof}[1][\protect\proofname]{\par
\normalfont\topsep6\p@\@plus6\p@\relax
\trivlist
\itemindent\parindent
\item[\hskip\labelsep
\scshape
#1]\ignorespaces
}{%
\endtrivlist\@endpefalse
}
\providecommand{\proofname}{Proof}
\fi
\theoremstyle{definition}
\newtheorem{example}[thm]{\protect\examplename}
\theoremstyle{remark}
\newtheorem{rem}[thm]{\protect\remarkname}

\makeatother

\usepackage{babel}
\providecommand{\examplename}{Example}
\providecommand{\lemmaname}{Lemma}
\providecommand{\remarkname}{Remark}
\providecommand{\theoremname}{Theorem}

\begin{document}

\title{Congruence conditions on the number of terms in sums of consecutive
squared integers equal to squared integers }

\author{Vladimir Pletser}

\maketitle

\lyxaddress{European Space Research and Technology Centre, ESA-ESTEC P.O. Box
299, NL-2200 AG Noordwijk, The Netherlands; E-mail: Vladimir.Pletser@esa.int}
\begin{abstract}
\noindent Considering the problem of finding all the integer solutions
of the sum of $M$ consecutive integer squares starting at $a^{2}$
being equal to a squared integer $s^{2}$, it is shown that this problem
has no solutions if $M\equiv3,5,6,7,8$ or $10\left(mod\,12\right)$
and has integer solutions if $M\equiv0,9,24$ or $33\left(mod\,72\right)$;
or $M\equiv1,2$ or $16\left(mod\,24\right)$; or $M\equiv11\left(mod\,12\right)$.
All the allowed values of $M$ are characterized using necessary conditions.
If $M$ is a square itself, then $M\equiv1\left(mod\,24\right)$ and
$\left(M-1\right)/24$ are all pentagonal numbers, except the first
two.
\end{abstract}


\section{Introduction}

Lucas stated in 1873 \cite{key-1-24} (see also \cite{key-1-25})
that$\left(1^{2}+...+n^{2}\right)$ is a square only for $n=1$ and
$24$. He proposed further in 1875 \cite{key-1-15} the well known
cannonball problem, namely to find a square number of cannonballs
stacked in a square pyramid. This problem can clearly be written as
a Diophantine equation $\sum_{i=1}^{M}\left(i^{2}\right)=M\left(M+1\right)\left(2M+1\right)/6=s^{2}$.
The only solutions are $s^{2}=1$ and $4900$, which correspond to
the sum of the first $M$ squared integers for $M=1$ and $M=24$.
This was partially proven by Moret-Blanc \cite{key-1-16} and Lucas
\cite{key-1-17}, and entirely proven later on by Watson \cite{key-1-18}
(with elementary proofs in most cases and using elliptic functions
for one case), Ljunggren \cite{key-1-21}, Ma \cite{key-1-19} and
Anglin \cite{key-1-20} (both with only elementary proofs). 

A more general problem is to find all values of $a$ for which the
sum of the $M$ consecutive integer squares starting from $a^{2}\geq1$
is itself an integer square $s^{2}$. Different approaches have been
proposed to solve this problem. Alfred studied \cite{key-1-22} several
necessary conditions on the values of $M$ (with the notations of
this paper), finding that $M=2,11,23,24,26,...$ until $M=500$ by
studying basic congruence equations of $M$, without being able to
conclude if there were solutions for $M=107,193,227,275,457$. This
was further addressed by Philipp \cite{key-1-23} who showed that
solutions exist for $M=107,193,457$ but not for $M=227,275$, and
proving that there are a finite or an infinite number of solutions
depending on whether $M$ is or not a square integer. Laub showed
\cite{key-1-26} that the set of values of $M$ yielding the sum of
$M$ consecutive squared integers being a squared integer, is infinite
and has density zero. Beeckmans demonstrated \cite{key-1-4} eight
necessary conditions on $M$ and gave a list of values of $M<1000$
with the corresponding smallest value of $a>0$, indicating two cases
for $M=25$ and $842$ complying with the eight necessary conditions
but not providing solutions to the problem.

In this paper, the method of determining the set of allowed values
of $M$ that yield the sum of $M$ consecutive squared integers to
be a squared integer is extended by expressing congruent $\left(mod\,12\right)$
conditions on $M$ using Beeckmans' necessary conditions \cite{key-1-4},
showing that $M$ cannot be congruent $\left(mod\,12\right)$ to $3,5,6,7,8$
or $10$ and must be congruent $\left(mod\,12\right)$ to $0,1,2,4,9$
or $11$, yielding that $M\equiv0,9,24$ or $33\left(mod\,72\right)$;
or $M\equiv1,2$ or $16\left(mod\,24\right)$; or $M\equiv11\left(mod\,12\right)$.
It is shown also that if $M$ is a square itself, then $M$ must be
congruent to $1\left(mod\,24\right)$ and $\left(M-1\right)/24$ are
all pentagonal numbers, except the first two.

Throughout the paper, the notation $A\left(mod\, B\right)\equiv C$
is equivalent to $A\equiv C\left(mod\, B\right)$ and $A\equiv C\left(mod\, B\right)\Rightarrow A=Bk+C$
means that, if $A\equiv C\left(mod\, B\right)$, then $\exists k\in\mathbb{Z}^{+}$
such that $A=Bk+C$. By convention, $\sum_{j=inf}^{sup}f\left(j\right)=0$
if $sup<inf$ .

\section{Congruent $\left(mod\,12\right)$ values of $M$}

A first theorem specifies the congruent $\left(mod\,12\right)$ values
that $M$ cannot take. In the demonstration of this theorem, several
numerical series are encountered and the following lemma shows that
these series take integer values for the indicated conditions.
\begin{lem}
For $n,\alpha\in\mathbb{Z}^{+}$ and $i,\delta\in\mathbb{Z}^{*}$:

(i) $\left[\left(3^{2\left(n-1\right)}-1\right)/4\right]$ and \textup{$\left[\left(3^{2n-1}+1\right)/4\right]\in\mathbb{Z}^{*}$,}
$\forall n$; 

\qquad{}furthermore, $\left[\left(3^{2n-1}+1\right)/4\right]\equiv1$
or $3\left(mod\,4\right)$ for $n\equiv1$ or $0\left(mod\,2\right)$;

(ii) $\left[\left(3^{2n-1}-2^{\alpha}+1\right)/12\right]=\left[2\left(\sum_{i=0}^{n-2}3^{2i}\right)-\left(2^{\alpha-2}-1\right)/3\right]\in\mathbb{Z}^{+}$,
$\forall n\geq2$, 

\qquad{}$\forall\alpha\equiv0\left(mod\,2\right)$, and 

\qquad{}$\left[\left(3^{2n-1}-5\times2^{\alpha}+1\right)/12\right]=\left[2\left(\sum_{i=0}^{n-2}3^{2i}\right)-\left(5\times2^{\alpha-2}-1\right)/3\right]\in\mathbb{Z}^{+}$ 

\qquad{}$\forall n\geq2$, $\forall\alpha\equiv1\left(mod\,2\right),\alpha>1$;

(iii) $\left[\left(3^{2n-1}\times13+1\right)/16\right]\in\mathbb{Z}^{+}$,
\textup{$\forall n\equiv0\left(mod\,2\right)$}; 

\qquad{}$\left[\left(3^{2n-1}\times37+1\right)/16\right]\in\mathbb{Z}^{+}$,
\textup{$\forall n\equiv1\left(mod\,2\right)$}; 

\qquad{}$\left[\left(3^{2n-1}\times25-11\right)/16\right]\in\mathbb{Z}^{+}$,
\textup{$\forall n\equiv1\left(mod\,2\right)$}; 

(iv) $\left[\left(3^{2n-1}\left(13+24\delta\right)-23\right)/32\right]\in\mathbb{Z}^{+}$
for $\delta=0$, \textup{$\forall n\equiv3\left(mod\,4\right)$};
for $\delta=1$, 

\qquad{}\textup{$\forall n\equiv0\left(mod\,4\right)$}; for $\delta=2$,
\textup{$\forall n\equiv1\left(mod\,4\right)$}; and for $\delta=3$\textup{,
$\forall n\equiv2\left(mod\,4\right)$.}\end{lem}
\begin{proof}
For $n,n^{\prime},\alpha\in\mathbb{Z}^{+}$ and $i,\delta\in\mathbb{Z}^{*}$

(i) immediate as $\forall n$, $3^{2\left(n-1\right)}\equiv1\left(mod\,8\right)$
and $3^{2n-1}\equiv3\left(mod\,4\right)$ (%
\footnote{For increasing $n$, the series $\left[\left(3^{2n-1}+1\right)/4\right]=1,7,61,547,4921,...$
is given in \cite{key-1-6}.%
}). Furthermore, 

(i.1) if $n\equiv1\left(mod\,2\right)$ $\Rightarrow$ $n=2n^{\prime}+1$,
assume that $\left[\left(3^{2n-1}+1\right)/4\right]\equiv1\left(mod\,4\right)$,
then $\left(3^{2n-1}+1\right)\equiv4\left(mod\,16\right)$ $\Rightarrow$
$\left(3\left(3^{4n^{\prime}}-1\right)\right)\equiv0\left(mod\,16\right)$,
which is the case as $\forall n^{\prime}$, $\left(3^{2n^{\prime}}+1\right)\equiv0\left(mod\,2\right)$
and $\left(3^{2n^{\prime}}-1\right)\equiv0\left(mod\,8\right)$; 

(i.2) if now $n\equiv0\left(mod\,2\right)$ $\Rightarrow$ $n=2n^{\prime}$,
assume that $\left[\left(3^{2n-1}+1\right)/4\right]\equiv3\left(mod\,4\right)$,
then $\left(3^{2n-1}+1\right)\equiv12\left(mod\,16\right)$ $\Rightarrow$
$\left(3\left(3^{4n^{\prime}-2}-1\right)\right)\equiv8\left(mod\,16\right)$,
which is the case as $\forall n^{\prime}$, $\left(3^{2n^{\prime}-1}+1\right)\equiv0\left(mod\,4\right)$
and $\left(3^{2n^{\prime}-1}-1\right)\equiv0\left(mod\,2\right)$. 

(ii) As $\forall n\geq2$, $\left[\left(3^{2\left(n-1\right)}-1\right)/8\right]=\sum_{i=0}^{n-2}3^{2i}$
(%
\footnote{For increasing $n$, the series $\left[\left(3^{2\left(n-1\right)}-1\right)/8\right]=1,10,91,820,7381,...$
is given in \cite{key-1-5,key-2-5}.%
}), then:
\begin{eqnarray}
\left(\frac{3^{2n-1}-2^{\alpha}+1}{12}\right) & = & \left(\frac{3^{2n-1}-3}{12}\right)-\left(\frac{2^{\alpha}-4}{12}\right)\nonumber \\
 & = & 2\left(\sum_{i=0}^{n-2}3^{2i}\right)-\left(\frac{2^{\alpha-2}-1}{3}\right)\in\mathbb{Z}^{+}\label{eq:7}
\end{eqnarray}
as $\forall\alpha\equiv0\left(mod\,2\right)$, $2^{\alpha-2}\equiv1\left(mod\,3\right)$,
and 
\begin{eqnarray}
\left(\frac{3^{2n-1}-5\times2^{\alpha}+1}{12}\right) & = & \left(\frac{3^{2n-1}-3}{12}\right)-\left(\frac{5\times2^{\alpha}-4}{12}\right)\nonumber \\
 & = & 2\left(\sum_{i=0}^{n-2}3^{2i}\right)-\left(\frac{5\times2^{\alpha-2}-1}{3}\right)\in\mathbb{Z}^{+}\label{eq:8}
\end{eqnarray}
 as $\forall\alpha\equiv1\left(mod\,2\right),\alpha>1$, $2^{\alpha-2}\equiv2\left(mod\,3\right)$
$\Rightarrow$ $\left(5\times2^{\alpha-2}\right)\equiv1\left(mod\,3\right)$.

(iii) Immediate as $\forall n\equiv0\left(mod\,2\right)$, $3^{2n}\equiv1\left(mod\,16\right)$
and $\forall n\equiv1\left(mod\,2\right)$, $3^{2n-1}\equiv3\left(mod\,16\right)$
$\Rightarrow$ $\left(3^{2n-1}\times5\right)\equiv15\left(mod\,16\right)$,
yielding:

(iii.1) $\left(3^{2n-1}\times13+1\right)\left(mod\,16\right)\equiv\left(-3^{2n}+1\right)\left(mod\,16\right)\equiv0$; 

(iii.2) $\left(3^{2n-1}\times37+1\right)\left(mod\,16\right)\equiv\left(3^{2n-1}\times5+1\right)\left(mod\,16\right)\equiv0$;

(iii.3) $\left(3^{2n-1}\times25-11\right)\left(mod\,16\right)\equiv$$\left(5\left(3^{2n-1}\times5+1\right)\right)\left(mod\,16\right)\equiv0$.

(iv) As $\left(3^{2n-1}\left(13+24\delta\right)-23\right)\left(mod\,32\right)\equiv3^{2}\left(3^{2n-3}\left(13+24\delta\right)+1\right)\left(mod\,32\right)$, 

$\delta=0$: $3^{2}\left(3^{2n-3}\times13+1\right)\left(mod\,32\right)\equiv0$
as $\forall n\equiv3\left(mod\,4\right)$, $3^{2n-3}\equiv27\left(mod\,32\right)$
$\Rightarrow$ $\left(3^{2n-3}\times13\right)\equiv31\left(mod\,32\right)$;

$\delta=1$: $3^{2}\left(3^{2n-3}\times37+1\right)\left(mod\,32\right)\equiv3^{2}\left(3^{2n-3}\times5+1\right)\left(mod\,32\right)\equiv0$
as $\forall n\equiv0\left(mod\,4\right)$, $3^{2n-3}\equiv19\left(mod\,32\right)$
$\Rightarrow$ $\left(3^{2n-3}\times5\right)\equiv31\left(mod\,32\right)$; 

$\delta=2$: $3^{2}\left(3^{2n-3}\times61+1\right)\left(mod\,32\right)\equiv3^{2}\left(-3^{2n-2}+1\right)\left(mod\,32\right)\equiv0$
as

$\forall n\equiv1\left(mod\,4\right)$, $3^{2n-2}\equiv1\left(mod\,32\right)$; 

$\delta=3$: $3^{2}\left(3^{2n-3}\times85+1\right)\left(mod\,32\right)\equiv3^{2}\left(3^{2n-2}\times7+1\right)\left(mod\,32\right)\equiv0$
as $\forall n\equiv2\left(mod\,4\right)$, $3^{2n-2}\equiv9\left(mod\,32\right)$
$\Rightarrow$ $\left(3^{2n-2}\times7\right)\equiv31\left(mod\,32\right)$.
\end{proof}
The following theorem can now be demonstrated with the eight necessary
conditions given by Beeckmans \cite{key-1-4} on the value of $M$
for (\ref{eq:53-3}) to hold, that can be summarized as follows, with
the notations of this paper and where $e,\alpha\in\mathbb{Z}^{+}$:

1) If $M\equiv0\left(mod\,2^{e}\right)$ or if $M\equiv0\left(mod\,3^{e}\right)$
or if $M\equiv-1\left(mod\,3^{e}\right)$, then $e\equiv1\left(mod\,2\right)$;
(C1.1, C1.2, C1.3)

2) If $p>3$ is prime, $M\equiv0\left(mod\, p^{e}\right)$, $e\equiv1\left(mod\,2\right)$,
then $p\equiv\pm1\left(mod\,12\right)$; (C2)

3) If $p\equiv3\left(mod\,4\right),p>3$ is prime, $M\equiv-1\left(mod\, p^{e}\right)$,
then $e\equiv0\left(mod\,2\right)$; (C3)

4) $M\neq3\left(mod\,9\right)$, $M\neq\left(2^{\alpha}-1\right)\left(mod\,2^{\alpha+2}\right)$
and $M\neq2^{\alpha}\left(mod\,2^{\alpha+2}\right)$ $\forall\alpha\geq2$.
(C4.1, C4.2, C4.3)
\begin{thm}
For $M>1,\in\mathbb{Z}^{+}$, the sum of squares of $M$ consecutive
integers cannot be an integer square if $M\equiv3,5,6,7,8$ or $10\left(mod\,12\right)$. 
\end{thm}
The demonstration is made in the order $M\equiv5,7,6,10,8$ and $3\left(mod\,12\right)$. 
\begin{proof}
For $M,\mu,i,k,K,m,m_{i},e_{i},p_{i},n,\alpha,\beta,\epsilon,\gamma_{n},\kappa,\xi,A,B\in\mathbb{Z}^{+},\eta\in\mathbb{Z}$,
$M>1$, $3\leq\mu\leq10$, let $M\equiv\mu\left(mod\,12\right)$ $\Rightarrow$
$M=12m+\mu$.

(i) For $\mu=5$ or $7$, $M=12m+5$ or $12m+7$, let $\prod\left(p_{i}^{e_{i}}\right)$
be the decomposition of $M$ in $i$ prime factors $p_{i}$, with
$\prod\left(p_{i}^{e_{i}}\right)\equiv5$ or $7\left(mod\,12\right)$.
Then one of the prime factors is $p_{j}\equiv5$ or $7\left(mod\,12\right)$
with an exponent $e_{j}\equiv1\left(mod\,2\right)$ (the remaining
co-factor is $\left(\prod\left(p_{i}^{e_{i}}\right)/p_{j}^{e_{j}}\right)\equiv1$
or $11\left(mod\,12\right)$), contradicting (C2) and these values
of $M$ must be rejected.

(ii) For $\mu=6$ or $10$, $M=12m+6$ or $12m+10$, $M+1=4\left(3m+1\right)+3$
or $4\left(3m+2\right)+3$, i.e. in both cases $\left(M+1\right)\equiv3\left(mod\,4\right)$.
Let $\prod\left(p_{i}^{e_{i}}\right)$ be the decomposition of $\left(M+1\right)$
in $i$ prime factors $p_{i}$. Then one of the prime factors is $p_{j}\equiv3\left(mod\,4\right)$
with an exponent $e_{j}\equiv1\left(mod\,2\right)$ (the remaining
co-factor being $\left(\prod\left(p_{i}^{e_{i}}\right)/p_{j}^{e_{j}}\right)\equiv1\left(mod\,4\right)$),
contradicting (C3).

(iv) For $\mu=8$, $M=12m+8$ and $M+1=3\left(4m+3\right)$, cases
appear cyclically with values of $\left(M+1\right)$ having either
a factor $3$ with an even exponent or a factor $f$ such as $f\equiv3\left(mod\,4\right)$.
Indeed, let first $m\neq0\left(mod\,3\right)$ and second $m\equiv0\left(mod\,3\right)$
$\Rightarrow$ $m=3m_{1}$. Let then first $m_{1}\neq2\left(mod\,3\right)$
and second $m_{1}\equiv2\left(mod\,3\right)$ $\Rightarrow$ $m_{1}=3m_{2}+2$.
Let then again first $m_{2}\neq0\left(mod\,3\right)$ and second $m_{2}\equiv0\left(mod\,3\right)$
$\Rightarrow$$m_{2}=3m_{3}$, and so on, yielding: 

$M+1=3\left(4m+3\right)$,

\qquad{}$\Rightarrow$ if $m\neq0\left(mod\,3\right)$ $\Rightarrow$
$\left(M+1\right)\equiv3\left(mod\,4\right)$,

\qquad{}$\Rightarrow$ if $m\equiv0\left(mod\,3\right)$ $\Rightarrow$
$m=3m_{1}$ $\Rightarrow$ $M+1=3^{2}\left(4m_{1}+1\right)$,

\qquad{}\qquad{}$\Rightarrow$ if $m_{1}\neq2\left(mod\,3\right)$
$\Rightarrow$ $\left(M+1\right)\equiv0\left(mod\,3^{2}\right)$,

\qquad{}\qquad{}$\Rightarrow$ if $m_{1}\equiv2\left(mod\,3\right)$
$\Rightarrow$ $m_{1}=3m_{2}+2$ $\Rightarrow$ $M+1=3^{3}\left(4m_{2}+3\right)$,

\qquad{}\qquad{}\qquad{}$\Rightarrow$ if $m_{2}\neq0\left(mod\,3\right)$
$\Rightarrow$ $\left(M+1\right)\equiv3\left(mod\,4\right)$,

\qquad{}\qquad{}\qquad{}$\Rightarrow$ if $m_{2}\equiv0\left(mod\,3\right)$
$\Rightarrow$ $m_{2}=3m_{3}$ $\Rightarrow$ $M+1=3^{4}\left(4m_{3}+1\right)$,

and so on. After $n$ iterations, $\exists m_{n}\in\mathbb{Z}^{+}$
such as either $M+1=3^{n}\left(4m_{n}+1\right)$ if $n\equiv0\left(mod\,2\right)$,
contradicting (C1.3), or $M+1=3^{n}\left(4m_{n}+3\right)$ if $n\equiv1\left(mod\,2\right)$.
Then let $\prod\left(p_{i}^{e_{i}}\right)$ be the decomposition of
$\left[\left(M+1\right)/3^{n}\right]$ in $i$ prime factors $p_{i}$,
with $\prod\left(p_{i}^{e_{i}}\right)\equiv3\left(mod\,4\right)$.
Then one of the prime factors is $p_{j}\equiv3\left(mod\,4\right)$
with an exponent $e_{j}$ such as $e_{j}\equiv1\left(mod\,2\right)$
(the remaining co-factor being such as $\left(\prod\left(p_{i}^{e_{i}}\right)/p_{j}^{e_{j}}\right)\equiv1\left(mod\,4\right)$),
contradicting (C3). Therefore, these values of $M$ must be rejected
in both cases.

(v) For $\mu=3$, $M=3\left(4m+1\right)$, cases appear cyclically
with values of $M$ being the product of a power of $3$ and a factor
which is $\left(mod\,12\right)$ congruent to either $1,5,7$ or $11$. 

(v.1) Let $m$ be successively $\left(mod\,3\right)$ congruent to
$0,1$ and $2$, and the $m\equiv2\left(mod\,3\right)$ step is subdivided
in $m\equiv5,8$ and $2\left(mod\,9\right)$ sub-steps; the process
is then repeated, yielding respectively:

$M=3\left(4m+1\right)$

\qquad{}$\Rightarrow$ if $m=3m_{1}$ $\Rightarrow$ $M=3\left(12m_{1}+1\right)$, 

\qquad{}$\Rightarrow$ if $m=3m_{1}+1$ $\Rightarrow$ $M=3\left(12m_{1}+5\right)$, 

\qquad{}$\Rightarrow$ if $m=3^{2}m_{1}+5$ $\Rightarrow$ $M=3^{2}\left(12m_{1}+7\right)$, 

\qquad{}$\Rightarrow$ if $m=3^{2}m_{1}+8$ $\Rightarrow$ $M=3^{2}\left(12m_{1}+11\right)$, 

\qquad{}$\Rightarrow$ if $m=3^{2}m_{1}+2$ $\Rightarrow$ $M=3^{3}\left(4m_{1}+1\right)$,

\qquad{}\qquad{}$\Rightarrow$ if $m_{1}=3m_{2}$ $\Rightarrow$
$m=3^{3}m_{2}+2$ $\Rightarrow$ $M=3^{3}\left(12m_{2}+1\right)$,

\qquad{}\qquad{}$\Rightarrow$ if $m_{1}=3m_{2}+1$ $\Rightarrow$
$m=3^{3}m_{2}+11$ $\Rightarrow$ $M=3^{3}\left(12m_{2}+5\right)$, 

\qquad{}\qquad{}$\Rightarrow$ if $m_{1}=3^{2}m_{2}+5$ $\Rightarrow$
$m=3^{4}m_{2}+47$ $\Rightarrow$ $M=3^{4}\left(12m_{2}+7\right)$, 

\qquad{}\qquad{}$\Rightarrow$ if $m_{1}=3^{2}m_{2}+8$ $\Rightarrow$
$m=3^{4}m_{2}+74$ $\Rightarrow$ $M=3^{4}\left(12m_{2}+11\right)$,

\qquad{}\qquad{}$\Rightarrow$ if $m_{1}=3^{2}m_{2}+2$ $\Rightarrow$
$m=3^{4}m_{2}+20$ $\Rightarrow$ $M=3^{5}\left(4m_{2}+1\right)$.

Taking again $\left(mod\,3\right)$ and $\left(mod\,9\right)$ congruent
values of $m_{2}$ yield new expressions of $M$ as a product of a
power of $3$ and a factor $\left(mod\,12\right)$ congruent to either
$1$,5,7 or $11$. One obtains then after $n$ iterations, with $\left[\left(3^{2\left(n-1\right)}-1\right)/4\right]\in\mathbb{Z}^{+}$
(see Lemma 1): 

\qquad{}if $m=3^{2n-1}m_{n}+\left[\left(3^{2\left(n-1\right)}-1\right)/4\right]$
$\Rightarrow$ $M=3^{2n-1}\left(12m_{n}+1\right)$, 

\qquad{}if $m=3^{2n-1}m_{n}+\left[3^{2\left(n-1\right)}+\left(3^{2\left(n-1\right)}-1\right)/4\right]$
$\Rightarrow$ $M=3^{2n-1}\left(12m_{n}+5\right)$,

\qquad{}if $m=3^{2n}m_{n}+\left[5\times3^{2\left(n-1\right)}+\left(3^{2\left(n-1\right)}-1\right)/4\right]$
$\Rightarrow$ $M=3^{2n}\left(12m_{n}+7\right)$, 

\qquad{}if $m=3^{2n}m_{n}+\left[8\times3^{2\left(n-1\right)}+\left(3^{2\left(n-1\right)}-1\right)/4\right]$
$\Rightarrow$ $M=3^{2n}\left(12m_{n}+11\right)$.

(v.2) For $M=3^{2n-1}\left(12m_{n}+5\right)$, let $\prod\left(p_{i}^{e_{i}}\right)$
be the decomposition of $\left(M/3^{2n-1}\right)$ in $i$ prime factors
$p_{i}$, with $\prod\left(p_{i}^{e_{i}}\right)\equiv5\left(mod\,12\right)$.
Then one of the prime factors is either $p_{j}\equiv5$ or $7\left(mod\,12\right)$
(the remaining co-factor being respectively either $\left(\prod\left(p_{i}^{e_{i}}\right)/p_{j}\right)\equiv1$
or $11\left(mod\,12\right)$), contradicting (C2).

(v.3) For $M=3^{2n}\left(12m_{n}+7\right)$ and $M=3^{2n}\left(12m_{n}+11\right)$,
both contradict (C1.2) as $\left(12m_{n}+7\right)$ and $\left(12m_{n}+11\right)$
cannot be $\left(mod\,3\right)$ congruent to $0$. 

(v.4) For $M=3^{2n-1}\left(12m_{n}+1\right)$, if $n=1$, $M=3\left(12m_{1}+1\right)$
contradicts (C4.1).

For $n\geq2$, (C4.2) is used first in (v.4.1) to reject some values
of $M$, then (C.3) is used in (v.4.2) to reject those values of $M$
that were not rejected by (C4.2).

(v.4.1) Condition (C4.2) for $M\equiv3\left(mod\,12\right)$ and $\alpha\geq2$
yields 

$M\neq\left(2^{\alpha}-1\right)\left(mod\left(3\times2^{\alpha+2}\right)\right)$
if $2^{\alpha}\equiv1\left(mod\,3\right)$, i.e. $\alpha\equiv0\left(mod\,2\right)$,
and

$M\neq\left(5\times2^{\alpha}-1\right)\left(mod\left(3\times2^{\alpha+2}\right)\right)$
if $2^{\alpha}\neq1\left(mod\,3\right)$, i.e. $\alpha\equiv1\left(mod\,2\right)$.

Those values of $m_{n}$ yielding $M=3^{2n-1}\left(12m_{n}+1\right)$
to be rejected are

\begin{equation}
\left(3^{2n-1}m_{n}\right)\equiv-\beta\left(mod\,2^{\alpha}\right)\label{eq:41-1}
\end{equation}
 with, for $\alpha\equiv0\left(mod\,2\right)$ and Lemma 1, 
\begin{equation}
\beta=\left(\frac{3^{2n-1}-2^{\alpha}+1}{12}\right)=2\left(\sum_{i=0}^{n-2}3^{2i}\right)-\left(\frac{2^{\alpha-2}-1}{3}\right)\label{eq:42-1}
\end{equation}
and, for $\alpha\equiv1\left(mod\,2\right)$ and Lemma 1,
\begin{equation}
\beta=\left(\frac{3^{2n-1}-5\times2^{\alpha}+1}{12}\right)=2\left(\sum_{i=0}^{n-2}3^{2i}\right)-\left(\frac{5\times2^{\alpha-2}-1}{3}\right)\label{eq:43-1}
\end{equation}
Then the values of $m_{n}$ yielding $M=3^{2n-1}\left(12m_{n}+1\right)$
to be rejected are

\begin{equation}
m_{n}\equiv m_{n0}\left(mod\,2^{\alpha}\right)\Rightarrow m_{n}=2^{\alpha}i+m_{n0}\label{eq:44-4}
\end{equation}
where $m_{n0}$ is the smallest value of $m_{n}$ for (\ref{eq:41-1})
to hold, i.e.$\exists m_{n0}\in\mathbb{Z}^{*}$ such as 
\begin{equation}
K=\left(\frac{3^{2n-1}m_{n0}+\beta}{2^{\alpha}}\right)\in\mathbb{Z}^{+}\label{eq:44-3}
\end{equation}
Table 1 shows the first values of $m_{n0}$.
\begin{table}
\caption{Values of $m_{n0}$ }

\centering{}%
\begin{tabular}{|c||c|c|c|c|c|}
\hline 
 & $n=2$ & $n=3$ & $n=4$ & $n=5$ & $n=6$\tabularnewline
\hline 
\hline 
$\alpha=2$ & 2 & 0 & 2 & 0 & 2\tabularnewline
\hline 
$\alpha=3$ & 3 & 5 & 7 & 1 & 3\tabularnewline
\hline 
$\alpha=4$ & 13 & 15 & 1 & 3 & 5\tabularnewline
\hline 
$\alpha=5$ & 17 & 3 & 5 & 23 & 25\tabularnewline
\hline 
$\alpha=6$ & 57 & 11 & 13 & 63 & 33\tabularnewline
\hline 
$\alpha=7$ & 73 & 27 & 93 & 15 & 49\tabularnewline
\hline 
$\alpha=8$ & 105 & 59 & 253 & 47 & 81\tabularnewline
\hline 
\end{tabular}
\end{table}
 For $\alpha=2$ and $\alpha=3$, the values of $m_{n0}$ repeat themselves.
Taking the $\left(mod\,2^{\alpha}\right)$ congruence of $3^{2n-1}$
and $\beta$ in (\ref{eq:44-3}) yield $\left[\left(\left(3^{2n-1}\left(mod\,2^{\alpha}\right)\right)m_{n0}+\beta\left(mod\,2^{\alpha}\right)\right)/2^{\alpha}\right]\in\mathbb{Z}^{+}$,
meaning that for $\alpha=2$ and $3$ and $\forall n$, $3^{2n-1}\equiv3\left(mod\,4\right)$
and $3\left(mod\,8\right)$. Furthermore, from (\ref{eq:42-1}), for
$\alpha=2$, $\beta=2\left(\sum_{i=0}^{n-2}3^{2i}\right)\equiv2$
or $0\left(mod\,4\right)$ for $n\equiv0$ or $1\left(mod\,2\right)$,
while for $\alpha=3$, $\beta=2\left(\sum_{i=0}^{n-2}3^{2i}\right)-3\equiv7,1,3$
or $5\left(mod\,8\right)$ respectively for $n\equiv2,3,0$ or $1\left(mod\,4\right)$.
Therefore, the values of $m_{n0}$ for $\alpha=2$ and $\alpha=3$
appear cyclically, respectively $m_{n0}=2$ and $0$ for $n\equiv0$
and $1\left(mod\,2\right)$, and $m_{n0}=3,5,7$ and $1$ for $n\equiv2,3,0$
and $1\left(mod\,4\right)$.

(v.4.2) Those values of $M=3^{2n-1}\left(12m_{n}+1\right)$ with $n\geq2$
that are not rejected by (C4.2) in the previous section (v.4), can
be rejected by (C3). It is sufficient to show as above that $\left(M+1\right)$
has a factor $f$ such as $f\equiv3\left(mod\,4\right)$, as the decomposition
of $f$ in product of prime factors includes then a prime factor $p_{j}^{e_{j}}$
such as $p_{j}^{e_{j}}\equiv3\left(mod\,4\right)$ with $e_{j}\equiv1\left(mod\,2\right)$.
One has then generally 
\begin{equation}
M+1=3^{2n-1}\left(12m_{n}+1\right)+1=4\left(3^{2n}m_{n}+\frac{3^{2n-1}+1}{4}\right)\label{eq:49-2}
\end{equation}
with $\left[\left(3^{2n-1}+1\right)/4\right]\equiv1$ or $3\left(mod\,4\right)$
for $n\equiv1$ or $0\left(mod\,2\right)$ (see Lemma 1). Let then
$\left[\left(3^{2n-1}+1\right)/4\right]=4\gamma_{n}+1$ or $4\gamma_{n}+3$
for $n\equiv1$ or $0\left(mod\,2\right)$.

(v.4.2.1) For an even number of iterations, i.e. $n\equiv0\left(mod\,2\right)$,
as the values of $m_{n}$ from (\ref{eq:44-4}) yielding $M=3^{2n-1}\left(12m_{n}+1\right)$
to be rejected for $\alpha=2$ are $m_{n}\equiv2\left(mod\,4\right)$,
let us show that $M=3^{2n-1}\left(12m_{n}+1\right)$ can also be rejected
by (C.3) for $m_{n}\equiv0,3$ and $1\left(mod\,4\right)$.

(v.4.2.1.1) Let first $m_{n}\equiv0\left(mod\,4\right)$ $\Rightarrow$
$m_{n}=4m_{n}^{\prime}$ and (\ref{eq:49-2}) yields $M+1=4\left[4\left(3^{2n}m_{n}^{\prime}+\gamma_{n}\right)+3\right]$,
contradicting (C3).

(v.4.2.1.2) Let now $m_{n}\equiv3\left(mod\,4\right)$ and two cases
are considered. 

First, as the values of $m_{n}$ to be rejected for $\alpha=3$ and
$\forall n\equiv0\left(mod\,4\right)$ are $m_{n}\equiv7\left(mod\,8\right)$,
let $m_{n}\equiv3\left(mod\,8\right)$ $\Rightarrow$ $m_{n}=8m_{n}^{\prime}+3$,
yielding with Lemma 1, 
\begin{eqnarray}
M+1 & = & 4\left(3^{2n}\times8m_{n}^{\prime}+3^{2n+1}+\frac{3^{2n-1}+1}{4}\right)\nonumber \\
 & = & 8\left[4\left(3^{2n}m_{n}^{\prime}+\frac{3^{2n-1}\times37-23}{32}\right)+3\right]\label{eq:17}
\end{eqnarray}
contradicting again (C3). 

Second, as the values of $m_{n}$ to be rejected for $\alpha=3$ and
$\forall n\equiv0\left(mod\,4\right)$ are $m_{n}\equiv3\left(mod\,8\right)$,
let $m_{n}\equiv7\left(mod\,8\right)$ $\Rightarrow$ $m_{n}=8m_{n}^{\prime}+7$,
yielding with Lemma 1, 
\begin{eqnarray}
M+1 & = & 4\left(3^{2n}\times8m_{n}^{\prime}+7\times3^{2n}+\frac{3^{2n-1}+1}{4}\right)\nonumber \\
 & = & 8\left[4\left(3^{2n}m_{n}^{\prime}+\frac{3^{2n-1}\times85-23}{32}\right)+3\right]\label{eq:18}
\end{eqnarray}
contradicting again (C3).

(v.4.2.1.3) Let now $m_{n}\equiv1\left(mod\,4\right)$ and consider
more generally the case 

\begin{equation}
m_{n}=4\left(2^{\kappa}m_{n}^{\prime}+\xi\right)+1=2^{\kappa+2}m_{n}^{\prime}+\left(4\xi+1\right)\label{eq:50-2}
\end{equation}
with $\kappa\geq2$ and $0\leq\xi\leq2^{\kappa}-1$, yielding from
(\ref{eq:49-2})
\begin{equation}
M+1=2^{\kappa+2}\left[4\left(3^{2n}m_{n}^{\prime}+A\right)+3\right]\label{eq:18-1}
\end{equation}
with $A=\left(3^{2n-1}\left(48\xi+13\right)-\left(3\times2^{\kappa+2}-1\right)\right)/2^{\kappa+4}$.
The values of $\xi\in\mathbb{Z}^{+}$ that renders $A\in\mathbb{Z}^{+}$
$\forall\kappa\geq2$ and $n\equiv0\left(mod\,2\right)$ are, with
Lemma 1, 
\begin{equation}
\xi=\left(\left(A+3\right)2^{\kappa-2}-\left(\frac{3^{2n-1}\times13+1}{16}\right)\right)3^{-2n}\label{eq:19}
\end{equation}
and shown in Table 2.
\begin{table}
\caption{Values of $\xi$ for $0\leq n<32$ with $n\equiv0\left(mod\,2\right)$
and $2\leq\kappa\leq10$}

\centering{}%
\begin{tabular}{|c||c|c|c|c|c|c|c|c|c|}
\hline 
$n$ & \multicolumn{9}{c|}{$\kappa$}\tabularnewline
\cline{2-10} 
 & 2 & 3 & 4 & 5 & 6 & 7 & 8 & 9 & 10\tabularnewline
\hline 
\hline 
0 & 0 & 3 & 1 & 13 & 5 & 53 & 21 & 213 & 85\tabularnewline
\hline 
2 & 1 & 0  & 6 & 2 & 58 & 42 & 138 & 74 & 970 \tabularnewline
\hline 
4 & 2 & 5 & 11 & 7 & 31 & 15 & 111 & 47 & 943\tabularnewline
\hline 
6 & 3 & 2 & 0 & 28 & 52 & 100 & 196 & 388 & 260\tabularnewline
\hline 
8 & 0 & 7 & 5 & 1 & 57 & 41 & 137 & 329 & 201\tabularnewline
\hline 
10 & 1 & 4 & 10 & 22 & 46 & 94 & 190 & 126 & 1022\tabularnewline
\hline 
12 & 2 & 1 & 15 & 27 & 19 & 3 & 99 & 35 & 931\tabularnewline
\hline 
14 & 3 & 6 & 4 & 16 & 40 & 24 & 120 & 312 & 184\tabularnewline
\hline 
16 & 0 & 3 & 9 & 21 & 45 & 29 & 253 & 189 & 61\tabularnewline
\hline 
18 & 1 & 0  & 14 & 10 & 34 & 18 & 242 & 434 & 818\tabularnewline
\hline 
20 & 2 & 5 & 3 & 15 & 7 & 119 & 87 & 279 & 663\tabularnewline
\hline 
22 & 3 & 2 & 8 & 4 & 28 & 76 & 44 & 492 & 876\tabularnewline
\hline 
24 & 0 & 7 & 13 & 9 & 33 & 17 & 113 & 305 & 689\tabularnewline
\hline 
26 & 1 & 4 & 2 & 30 & 22 & 70 & 38 & 486 & 358\tabularnewline
\hline 
28 & 2 & 1 & 7 & 3 & 59 & 107 & 75 & 267 & 139\tabularnewline
\hline 
30 & 3 & 6 & 12 & 24 & 16 & 0 & 224 & 416 & 288\tabularnewline
\hline 
\end{tabular}
\end{table}
 These values of $\xi$ can be represented as polynomials in $n^{\prime}=n/2$
with the independent term (i.e. the value of $\xi$ for $n\equiv0\left(mod\,2^{\kappa+1}\right)$)
being either $\sum_{i=0}^{\left(\kappa-4\right)/2}\left(2^{2i}\right)$
if $\kappa\equiv0\left(mod\,2\right)$ or $\left(2^{\kappa-2}+\sum_{i=0}^{\left(\kappa-3\right)/2}\left(2^{2i}\right)\right)$
if $\kappa\equiv1\left(mod\,2\right)$. The coefficients $c_{i}$
of the powers of $n^{\prime}$ in the polynomials $P\left(n^{\prime}\right)=\sum c_{i}n^{\prime}{}^{i}$
can be chosen to fit the congruence $\xi\equiv P\left(n^{\prime}\right)\left(mod\,2^{\kappa}\right)$
and the polynomials with the smallest $c_{i}\in\mathbb{Z}^{+}$ are
shown in Table 3.
\begin{table}
\caption{Polynomials $P\left(n^{\prime}\right)$ such as $\xi\equiv P\left(n^{\prime}\right)\left(mod\,2^{\kappa}\right)$
with $n^{\prime}=n/2$}

\centering{}%
\begin{tabular}{|c|c|}
\hline 
$\kappa$ & $P\left(n^{\prime}\right)\left(mod\,2^{\kappa}\right)$\tabularnewline
\hline 
\hline 
2 & $\left(n^{\prime}\right)\left(mod\,4\right)$\tabularnewline
\hline 
3 & $\left(5n^{\prime}+3\right)\left(mod\,8\right)$\tabularnewline
\hline 
4 & $\left(5n^{\prime}+1\right)\left(mod\,16\right)$\tabularnewline
\hline 
5 & $\left(8n^{\prime}{}^{2}+13n^{\prime}+13\right)\left(mod\,32\right)$\tabularnewline
\hline 
6 & $\left(24n^{\prime}{}^{2}+29n^{\prime}+15\right)\left(mod\,64\right)$\tabularnewline
\hline 
7 & $\left(56n^{\prime}{}^{2}+61n^{\prime}+53\right)\left(mod\,128\right)$\tabularnewline
\hline 
8 & $\left(56n^{\prime}{}^{2}+61n^{\prime}+21\right)\left(mod\,256\right)$\tabularnewline
\hline 
9 & $\left(384n^{\prime}{}^{3}+440n^{\prime}{}^{2}+61n^{\prime}+213\right)\left(mod\,512\right)$\tabularnewline
\hline 
10 & $\left(384n^{\prime}{}^{3}+952n^{\prime}{}^{2}+573n^{\prime}+85\right)\left(mod\,1024\right)$\tabularnewline
\hline 
\end{tabular}
\end{table}
 Replacing these values of $\xi$ in $m_{n}=2^{\kappa+2}m_{n}^{\prime}+\left(4\xi+1\right)$
yields $\left(M+1\right)$ to have a factor $f$ such as $f\equiv3\left(mod\,4\right)$,
contradicting again (C3). Therefore, all values of $n\equiv0\left(mod\,2\right)$
yield $M$ to be rejected.

(v.4.2.2) For an odd number of iterations, i.e. $n\equiv1\left(mod\,2\right)$,
as the values of $m_{n}$ from (\ref{eq:44-4}) yielding $M=3^{2n-1}\left(12m_{n}+1\right)$
to be rejected for $\alpha=2$ are $m_{n}\equiv0\left(mod\,4\right)$,
let us show that $M=3^{2n-1}\left(12m_{n}+1\right)$ can also be rejected
by (C.3) for $m_{n}\equiv2,1$ and $3\left(mod\,4\right)$.

(v.4.2.2.1) Let first $m_{n}\equiv2\left(mod\,4\right)$ $\Rightarrow$
$m_{n}=4m_{n}^{\prime}+2$ and (\ref{eq:49-2}) yields then 
\begin{eqnarray}
M+1 & = & 4\left(3^{2n}m_{n}+\gamma_{n}+1\right)\nonumber \\
 & = & 4\left[4\left(3^{2n}m_{n}^{\prime}+\frac{3^{2n-1}\times25-11}{16}\right)+3\right]\label{eq:20}
\end{eqnarray}
with Lemma 1, contradicting again (C3). 

(v.4.2.2.2) Let now $m_{n}\equiv1\left(mod\,4\right)$ and two cases
are again considered. 

First, as the values of $m_{n}$ to be rejected for $\alpha=3$ and
$\forall n\equiv3\left(mod\,4\right)$ are $m_{n}\equiv5\left(mod\,8\right)$,
let $m_{n}\equiv1\left(mod\,8\right)$ $\Rightarrow$ $m_{n}=8m_{n}^{\prime}+1$,
yielding with Lemma 1, 
\begin{eqnarray}
M+1 & = & 4\left(3^{2n}\times8m_{n}^{\prime}+3^{2n}+\frac{3^{2n-1}+1}{4}\right)\nonumber \\
 & = & 8\left[4\left(3^{2n}m_{n}^{\prime}+\frac{3^{2n-1}\times13-23}{32}\right)+3\right]\label{eq:21-1}
\end{eqnarray}
contradicting again (C3). 

Second, as the values of $m_{n}$ to be rejected for $\alpha=3$ and
$\forall n\equiv1\left(mod\,4\right)$ are $m_{n}\equiv1\left(mod\,8\right)$,
let $m_{n}\equiv5\left(mod\,8\right)$ $\Rightarrow$ $m_{n}=8m_{n}^{\prime}+5$,
yielding with Lemma 1, 
\begin{eqnarray}
M+1 & = & 4\left(3^{2n}\times8m_{n}^{\prime}+5\times3^{2n}+\frac{3^{2n-1}+1}{4}\right)\nonumber \\
 & = & 8\left[4\left(3^{2n}m_{n}^{\prime}+\frac{3^{2n-1}\times61-23}{32}\right)+3\right]\label{eq:22}
\end{eqnarray}
contradicting again (C3).

(v.4.2.2.3) Let now $m_{n}\equiv3\left(mod\,4\right)$ and consider
more generally the case 

\begin{equation}
m_{n}=4\left(2^{\kappa}m_{n}^{\prime}+\xi\right)+3=2^{\kappa+2}m_{n}^{\prime}+\left(4\xi+3\right)\label{eq:50-2-1}
\end{equation}
with $\kappa\geq2$ and $0\leq\xi\leq2^{\kappa}-1$, yielding
\begin{equation}
M+1=2^{\kappa+2}\left[4\left(3^{2n}m_{n}^{\prime}+B\right)+3\right]\label{eq:23}
\end{equation}
with $B=\left(3^{2n-1}\left(48\xi+37\right)-\left(3\times2^{\kappa+2}-1\right)\right)/2^{\kappa+4}$.
The values of $\xi\in\mathbb{Z}^{+}$ that renders $B\in\mathbb{Z}^{+}$
$\forall\kappa\geq2$ and $n\equiv1\left(mod\,2\right)$ are, with
Lemma 1,
\begin{equation}
\xi=\left(\left(B+3\right)2^{\kappa-2}-\left(\frac{3^{2n-1}\times37+1}{16}\right)\right)3^{-2n}\label{eq:25}
\end{equation}
and shown in Table 4\textbf{.}
\begin{table}
\caption{Values of $\xi$ for $0<n<32$ with $n\equiv1\left(mod\,2\right)$
and $2\leq\kappa\leq10$}

\centering{}%
\begin{tabular}{|c||c|c|c|c|c|c|c|c|c|}
\hline 
$n$ & \multicolumn{9}{c|}{$\kappa$}\tabularnewline
\cline{2-10} 
 & 2 & 3 & 4 & 5 & 6 & 7 & 8 & 9 & 10\tabularnewline
\hline 
\hline 
1 & 0 & 7 & 13 & 9 & 33 & 81 & 49 & 497 & 881\tabularnewline
\hline 
3 & 1 & 4 & 10 & 22 & 46 & 94 & 62 & 254 & 638\tabularnewline
\hline 
5 & 2 & 1 & 7 & 19 & 43 & 91 & 59 & 251 & 635\tabularnewline
\hline 
7 & 3 & 6 & 4 & 0 & 24 & 72 & 40 & 232 & 104\tabularnewline
\hline 
9 & 0 & 3 & 1 & 29 & 53 & 37 & 5 & 453 & 325\tabularnewline
\hline 
11 & 1 & 0 & 14 & 10 & 2 & 114 & 210 & 146 & 530\tabularnewline
\hline 
13 & 2 & 5 & 11 & 7 & 63 & 47 & 143 & 79 & 975\tabularnewline
\hline 
15 & 3 & 2 & 8 & 20 & 44 & 92 & 60 & 508 & 892\tabularnewline
\hline 
17 & 0 & 7 & 5 & 17 & 9 & 121 & 217 & 153 & 537\tabularnewline
\hline 
19 & 1 & 4 & 2 & 30 & 22 & 6 & 102 & 294 & 166\tabularnewline
\hline 
21 & 2 & 1 & 15 & 27 & 19 & 3 & 227 & 163 & 35\tabularnewline
\hline 
23 & 3 & 6 & 12 & 8 & 0 & 112 & 80 & 16 & 400\tabularnewline
\hline 
25 & 0 & 3 & 9 & 5 & 29 & 77 & 173 & 109 & 493\tabularnewline
\hline 
27 & 1 & 0 & 6 & 18 & 42 & 26 & 250 & 186 & 570\tabularnewline
\hline 
29 & 2 & 5 & 3 & 15 & 39 & 87 & 55 & 503 & 887\tabularnewline
\hline 
31 & 3 & 2 & 0 & 28 & 20 & 4 & 100 & 292 & 676\tabularnewline
\hline 
\end{tabular}
\end{table}
 These values of $\xi$ can be represented as polynomials in $n^{\prime}=\left(n-1\right)/2$
with the independent term (i.e. the value of $\xi$ for $n\equiv1\left(mod\,2^{\kappa+1}\right)$)
being either $\left(\sum_{i=0}^{\left(\kappa-4\right)/2}\left(2^{2i}\right)+4\sigma_{\kappa/2}\right)\left(mod\,2^{\kappa}\right)$
or 

$\left(\sum_{i=0}^{\left(\kappa-3\right)/2}\left(2^{2i}\right)+4\sigma_{\left(\kappa-1\right)/2}+2^{\kappa-2}\right)\left(mod\,2^{\kappa}\right)$
if respectively $\kappa\equiv0$ or $1\left(mod\,2\right)$, where
the integer sequence $\sigma_{j}=1,3,7,71,199,...$ is given in \cite{key-1-7,key-1-8}.
The coefficients $c_{i}$ of the powers of $n^{\prime}$ in the polynomials
$P\left(n^{\prime}\right)=\sum c_{i}n^{\prime}{}^{i}$ can also be
chosen to fit the congruence $\xi\equiv P\left(n^{\prime}\right)\left(mod\,2^{\kappa}\right)$
and the polynomials with the smallest $c_{i}\in\mathbb{Z}^{+}$ are
shown in Table 5.
\begin{table}
\caption{Polynomials $P\left(n^{\prime}\right)$ such as $\xi\equiv P\left(n^{\prime}\right)\left(mod\,2^{\kappa}\right)$
with $n^{\prime}=\left(n-1\right)/2$ }

\centering{}%
\begin{tabular}{|c|c|}
\hline 
$\kappa$ & $P\left(n^{\prime}\right)\left(mod\,2^{\kappa}\right)$\tabularnewline
\hline 
\hline 
2 & $\left(n^{\prime}\right)\left(mod\,4\right)$\tabularnewline
\hline 
3 & $\left(5n^{\prime}+7\right)\left(mod\,8\right)$\tabularnewline
\hline 
4 & $\left(13n^{\prime}+13\right)\left(mod\,16\right)$\tabularnewline
\hline 
5 & $\left(8n^{\prime}{}^{2}+5n^{\prime}+9\right)\left(mod\,32\right)$\tabularnewline
\hline 
6 & $\left(24n^{\prime}{}^{2}+53n^{\prime}+33\right)\left(mod\,64\right)$\tabularnewline
\hline 
7 & $\left(56n^{\prime}{}^{2}+85n^{\prime}+81\right)\left(mod\,128\right)$\tabularnewline
\hline 
8 & $\left(120n^{\prime}{}^{2}+149n^{\prime}+49\right)\left(mod\,256\right)$\tabularnewline
\hline 
9 & $\left(384n^{\prime}{}^{3}+504n^{\prime}{}^{2}+405n^{\prime}+457\right)\left(mod\,512\right)$\tabularnewline
\hline 
10 & $\left(384n^{\prime}{}^{3}+1016n^{\prime}{}^{2}+425n^{\prime}+881\right)\left(mod\,1024\right)$\tabularnewline
\hline 
\end{tabular}
\end{table}
 Replacing these values of $\xi$ in $m_{n}=2^{\kappa+2}m_{n}^{\prime}+\left(4\xi+3\right)$
yields $\left(M+1\right)$ to have a factor $f$ such as $f\equiv3\left(mod\,4\right)$,
contradicting again (C3). 

Therefore, all values of $n\equiv1\left(mod\,2\right)$ yield $M$
to be rejected.

(v.5) It follows that the sum of squares of $M$ consecutive integers
cannot be an integer square if $M\equiv3,5,6,7,8$ or $10\left(mod\,12\right)$.\end{proof}
\begin{example}
For an even number of iterations $n\equiv0\left(mod\,2\right)$ in
the case (v.4.2.1.3) above, the following example for $n=2$ shows
that there are no $m_{n}\equiv1\left(mod\,4\right)$ values such that
the sum of squares of $M=\left(12m_{n}+3\right)$ consecutive integers
can be an integer square as the following values of $m_{n}$ have
to be rejected:
\end{example}
$m_{n}\equiv1\left(mod\,32\right)$, i.e, $1,33,65,...$, by (C3),
$\kappa=3$, $\xi=0$ in (\ref{eq:50-2});

$m_{n}\equiv5\left(mod\,16\right)$, i.e. $5,21,37,...$, by (C3),
$\kappa=2$, $\xi=1$ in (\ref{eq:50-2});

$m_{n}\equiv9\left(mod\,128\right)$, i.e. $9,137,245,...$, by (C3),
$\kappa=5$, $\xi=2$ in (\ref{eq:50-2});

$m_{n}\equiv13\left(mod\,16\right)$, i.e. $13,29,45,...$, by (C4.2),
$\alpha=4$, $m_{n0}=13$ in (\ref{eq:44-4});

$m_{n}\equiv17\left(mod\,32\right)$, i.e. $17,49,81,...$, by (C4.2),
$\alpha=5$, $m_{n0}=17$ in (\ref{eq:44-4});

$m_{n}\equiv25\left(mod\,64\right)$, i.e. $25,89,153,...$, by (C3),
$\kappa=4$, $\xi=6$ in (\ref{eq:50-2});

$m_{n}\equiv41\left(mod\,1024\right)$, i.e. $41,1065,...$, by (C4.2),
$\alpha=10$, $m_{n0}=41$ in (\ref{eq:44-4});

$m_{n}\equiv57\left(mod\,64\right)$, i.e. $57,121,185,...$, by (C4.2),
$\alpha=6$, $m_{n0}=57$ in (\ref{eq:44-4}); etc.

For an odd number of iterations (i.e. $n\equiv1\left(mod\,2\right)$)
in the case (v.4.2.2.3) above, the following example for $n=3$ shows
that there are no $m_{n}\equiv3\left(mod\,4\right)$ values such that
the sum of squares of $M=\left(12m_{n}+3\right)$ consecutive integers
can be an integer square as the following values of $m_{n}$ have
to be rejected:

$m_{n}\equiv3\left(mod\,32\right)$, i.e. $3,35,67,...$, by (C4.2),
$\alpha=5$, $m_{n0}=3$ in (\ref{eq:44-4});

$m_{n}\equiv7\left(mod\,16\right)$, i.e. $7,23,39,...$, by (C3),
$\kappa=2$, $\xi=1$ in (\ref{eq:50-2-1});

$m_{n}\equiv11\left(mod\,64\right)$, i.e. $11,75,139,...$, by (C4.2),
$\alpha=6$, $m_{n0}=11$ in (\ref{eq:44-4});

$m_{n}\equiv15\left(mod\,16\right)$, i.e. $15,31,47,...$, by (C4.2),
$\alpha=4$, $m_{n0}=15$ in (\ref{eq:44-4});

$m_{n}\equiv19\left(mod\,32\right)$, i.e. $19,51,83,...$, by (C3),
$\kappa=3$, $\xi=4$ in (\ref{eq:50-2-1}); 

$m_{n}\equiv27\left(mod\,128\right)$, i.e. $27,155,283,...$, by
(C4.2), $\alpha=7$, $m_{n0}=27$ in (\ref{eq:44-4});

$m_{n}\equiv43\left(mod\,64\right)$, i.e. $43,107,171,...$, by (C3),
$\kappa=4$, $\xi=10$ in (\ref{eq:50-2-1}); etc.
\begin{rem}
Note that the values of $m_{n0}$ in section (v.4.1) above are not
independent and within the same $n^{th}$ iteration, the value $m_{n0,\alpha}$
of $m_{n0}$ for a given value of $\alpha$ is related to the preceding
value $m_{n0,\left(\alpha-1\right)}$ for $\left(\alpha-1\right)$
by 

\begin{equation}
m_{n0,\alpha}\left(mod\,2^{\alpha}\right)\equiv\left(m_{n0,\left(\alpha-1\right)}+\epsilon\times2^{\alpha-1}+2^{\alpha-3}\right)\label{eq:44-2}
\end{equation}
with either $\epsilon=-1$, or $0$, or $+1$. From (\ref{eq:44-3}),
$m_{n0}=\left(2^{\alpha}K-\beta\right)/3^{2n-1}$ and one has respectively
\begin{eqnarray}
m_{n0} & = & \frac{2^{\alpha}K-2\left(\sum_{i=0}^{n-2}3^{2i}\right)+\left(\frac{2^{\alpha-2}-1}{3}\right)}{3^{2n-1}}\,\,\textnormal{if}\,\,\alpha\equiv0\left(mod\,2\right)\label{eq:46}\\
m_{n0} & = & \frac{2^{\alpha}K-2\left(\sum_{i=0}^{n-2}3^{2i}\right)+\left(\frac{5\times2^{\alpha-2}-1}{3}\right)}{3^{2n-1}}\,\,\textnormal{if}\,\,\alpha\equiv1\left(mod\,2\right)\label{eq:47-1}
\end{eqnarray}
Forming now the difference $m_{n0,\alpha}-m_{n0,\left(\alpha-1\right)}$,
one obtains $m_{n0,\alpha}-m_{n0,\left(\alpha-1\right)}=\epsilon\times2^{\alpha-1}+2^{\alpha-3}$
with
\begin{equation}
\epsilon=\frac{\left(2K_{\alpha}-K_{\left(\alpha-1\right)}\right)-\left(\frac{3^{2n-1}+1}{4}+\eta\right)}{3^{2n-1}}\label{eq:48}
\end{equation}
with $\eta=0$ or $-1$ if $\alpha\equiv0$ or $1\left(mod\,2\right)$
and where $K_{\alpha}$ and $K_{\left(\alpha-1\right)}$ are the values
of $K$ corresponding to $m_{n0,\alpha}$ and $m_{n0,\left(\alpha-1\right)}$
in (\ref{eq:44-3}). $\epsilon=1$ or $0$ or $-1$ if $2K_{\alpha}-K_{\left(\alpha-1\right)}=\left(\eta+\left(5\times3^{2n-1}+1\right)/4\right)$
or $\left(\eta+\left(3^{2n-1}+1\right)/4\right)$ or $\left(\eta+\left(1-3^{2n}\right)/4\right)$
(see also Lemma 1).
\end{rem}
The next theorem gives additional conditions on the allowed $\left(mod\,12\right)$
congruent values that $M$ can take.
\begin{thm}
For $M>1,a,s\in\mathbb{Z}^{+}$, $i\in\mathbb{Z}^{*}$, there exist
$M$ satisfying $M\equiv0,1,2,4,9$ or $11\left(mod\,12\right)$ such
as the sums of $M$ consecutive squared integers $\left(a+i\right)^{2}$
equal integer squares $s^{2}$. Furthermore, if $M\equiv0\left(mod\,12\right)$,
then $M\equiv0$ or $24\left(mod\,72\right)$; if $M\equiv1\left(mod\,12\right)$,
then $M\equiv1\left(mod\,24\right)$; if $M\equiv2\left(mod\,12\right)$,
then $M\equiv2\left(mod\,24\right)$; if $M\equiv4\left(mod\,12\right)$,
then $M\equiv16\left(mod\,24\right)$; if $M\equiv9\left(mod\,12\right)$,
then $M\equiv9$ or $33\left(mod\,72\right)$; and the corresponding
congruent values of $a$ and $s$ are given in Table 6.
\begin{table}
\begin{centering}
\caption{Congruent values of $M$, $m$, $a$, and $s$}

\par\end{centering}

\centering{}%
\begin{tabular}{|c|c|c|c|c|}
\hline 
$\mu$ & $M\equiv$ & $m\equiv$ & $a\equiv$ & $s\equiv$\tabularnewline
\hline 
\hline 
$0$ & $0\left(mod\,72\right)$ & $0\left(mod\,6\right)$ & $\forall$ & $0\left(mod\,6\right)$\tabularnewline
\cline{2-5} 
 & $24\left(mod\,72\right)$ & $2\left(mod\,6\right)$ & $\forall$ & $2$ or $4\left(mod\,6\right)$\tabularnewline
\hline 
$1$ & $1\left(mod\,24\right)$ & $0\left(mod\,6\right)$ & $\forall$ & $\forall$\tabularnewline
\cline{3-5} 
 &  & $2\left(mod\,6\right)$ & $0\left(mod\,6\right)$ & $2$ or $4\left(mod\,6\right)$\tabularnewline
\cline{4-5} 
 &  &  & $3\left(mod\,6\right)$ & $1$ or $5\left(mod\,6\right)$\tabularnewline
\cline{3-5} 
 &  & $4\left(mod\,6\right)$ & $1\left(mod\,2\right)$ & $3\left(mod\,6\right)$\tabularnewline
\cline{4-5} 
 &  &  & $0\left(mod\,2\right)$ & $0\left(mod\,6\right)$\tabularnewline
\hline 
$2$ & $2\left(mod\,24\right)$ & $0\left(mod\,6\right)$ & $0,2\left(mod\,3\right)$ & $1$ or $5\left(mod\,6\right)$\tabularnewline
\cline{3-5} 
 &  & $2\left(mod\,6\right)$ & $1\left(mod\,3\right)$ & $3\left(mod\,6\right)$\tabularnewline
\cline{3-5} 
 &  & $4\left(mod\,6\right)$ & $\forall$ & $1\left(mod\,2\right)$\tabularnewline
\hline 
$4$ & $16\left(mod\,24\right)$ & $1\left(mod\,6\right)$ & $0\left(mod\,3\right)$ & $2$ or $4\left(mod\,6\right)$\tabularnewline
\cline{3-5} 
 &  & $3\left(mod\,6\right)$ & $1,2\left(mod\,3\right)$ & $0\left(mod\,6\right)$\tabularnewline
\cline{3-5} 
 &  & $5\left(mod\,6\right)$ & $\forall$ & $0\left(mod\,2\right)$\tabularnewline
\hline 
$9$ & $9\left(mod\,72\right)$ & $0\left(mod\,6\right)$ & $0\left(mod\,2\right)$ & $0\left(mod\,6\right)$\tabularnewline
\cline{4-5} 
 &  &  & $1\left(mod\,2\right)$ & $3\left(mod\,6\right)$\tabularnewline
\cline{2-5} 
 & $33\left(mod\,72\right)$ & $2\left(mod\,6\right)$ & $0\left(mod\,2\right)$ & $2$ or $4\left(mod\,6\right)$\tabularnewline
\cline{4-5} 
 &  &  & $1\left(mod\,2\right)$ & $1$ or $5\left(mod\,6\right)$\tabularnewline
\hline 
$11$ & $11\left(mod\,12\right)$ & $0\left(mod\,6\right)$ & $0,2\left(mod\,6\right)$ & $1$ or $5\left(mod\,6\right)$\tabularnewline
\cline{3-5} 
 &  & $1\left(mod\,6\right)$ & $1\left(mod\,6\right)$ & $2$ or $4\left(mod\,6\right)$\tabularnewline
\cline{4-5} 
 &  &  & $3,5\left(mod\,6\right)$ & $0\left(mod\,6\right)$\tabularnewline
\cline{3-5} 
 &  & $2\left(mod\,6\right)$ & $4\left(mod\,6\right)$ & $3\left(mod\,6\right)$\tabularnewline
\cline{3-5} 
 &  & $3\left(mod\,6\right)$ & $3,5\left(mod\,6\right)$ & $2$ or $4\left(mod\,6\right)$\tabularnewline
\cline{3-5} 
 &  & $4\left(mod\,6\right)$ & $0,2\left(mod\,6\right)$ & $3\left(mod\,6\right)$\tabularnewline
\cline{4-5} 
 &  &  & $4\left(mod\,6\right)$ & $1$ or $5\left(mod\,6\right)$\tabularnewline
\cline{3-5} 
 &  & $5\left(mod\,6\right)$ & $1\left(mod\,6\right)$ & $0\left(mod\,6\right)$\tabularnewline
\hline 
\end{tabular}
\end{table}
\end{thm}
\begin{proof}
For $M>1,m,a,s\in\mathbb{Z}^{+}$ and $\mu,i\in\mathbb{Z}^{*}$, $0\leq\mu\leq11$,
let $M\equiv\mu\left(mod\,12\right)$ $\Rightarrow$ $M=12m+\mu$. 

Expressing the sum of $M$ consecutive integer squares starting from
$a^{2}$ equal to an integer square $s^{2}$ as

\begin{equation}
\sum_{i=0}^{M-1}\left(a+i\right)^{2}=M\left[\left(a+\frac{M-1}{2}\right)^{2}+\frac{M^{2}-1}{12}\right]=s^{2}\label{eq:53-3}
\end{equation}

and replacing $M$ by $12m+\mu$ in (\ref{eq:53-3}) yields

\begin{equation}
\left(12m+\mu\right)\left[a^{2}+a\left(12m+\mu-1\right)+48m^{2}+2m\left(4\mu-3\right)+\frac{2\mu^{2}-3\mu+1}{6}\right]=s^{2}\label{eq:53-2}
\end{equation}

Recalling that integer squares are congruent to either $0,1,4$ or
$9\left(mod\,12\right)$, replacing the values of $\mu=0,1,2,4,9,11$
in (\ref{eq:53-2}) and reducing $\left(mod\,12\right)$ yield:

(i) for $\mu=0$, $\left(2m\left(6a^{2}-6a+1\right)-s^{2}\right)\equiv0\left(mod\,12\right)$. 

As $\forall a$, $\left(6a^{2}-6a\right)\equiv0\left(mod\,12\right)$,
it yields $\left(2m-s^{2}\right)\equiv0\left(mod\,12\right)$ $\Rightarrow$
$s\equiv0\left(mod\,6\right)$ for $m\equiv0\left(mod\,6\right)$
and $s\equiv2$ or $4\left(mod\,6\right)$ for $m\equiv2\left(mod\,6\right)$. 

Therefore, $M\equiv0\left(mod\,72\right)$ with $s\equiv0\left(mod\,6\right)$
or $M\equiv24\left(mod\,72\right)$ with $s\equiv2$ or $4\left(mod\,6\right)$
and $a$ can take any value.

(ii) for $\mu=1$, $\left(a^{2}+2m-s^{2}\right)\equiv0\left(mod\,12\right)$. 

For $a^{2}\equiv\left\{ 0,1,4,9\right\} \left(mod\,12\right)$, $2m\equiv\left\{ \left(0\,\textnormal{or}\,4\right),\left(0\,\textnormal{or}\,8\right),\left(0\,\textnormal{or}\,8\right),\left(0\,\textnormal{or}\,4\right)\right\} \left(mod\,12\right)$
respectively for $s^{2}\equiv\left\{ \left(0\,\textnormal{or}\,4\right),\left(1\,\textnormal{or}\,9\right),\left(4\,\textnormal{or}\,0\right),\left(9\,\textnormal{or}\,1\right)\right\} \left(mod\,12\right)$,
yielding 

$m\equiv0\left(mod\,2\right)$ and $M\equiv1\left(mod\,24\right)$.
Furthermore, 

- if $m\equiv0\left(mod\,6\right)$, $a$ and $s$ can take any values; 

- if $m\equiv2\left(mod\,6\right)$, either $a\equiv0\left(mod\,6\right)$
and $s\equiv2$ or $4\left(mod\,6\right)$, or $a\equiv3\left(mod\,6\right)$
and $s\equiv1$ or $5\left(mod\,6\right)$; and 

- if $m\equiv4\left(mod\,6\right)$, either $a\equiv1\left(mod\,2\right)$
and $s\equiv3\left(mod\,6\right)$, or $a\equiv0\left(mod\,2\right)$
and $s\equiv0\left(mod\,6\right)$.

(iii) for $\mu=2$, $\left(2\left(a^{2}+a\right)+2m+1-s^{2}\right)\equiv0\left(mod\,12\right)$. 

For $\left(2\left(a^{2}+a\right)+1\right)\equiv\left\{ 1,5\right\} \left(mod\,12\right)$,
$2m\equiv\left\{ \left(0\,\textnormal{or}\,8\right),\left(4\,\textnormal{or}\,8\right)\right\} \left(mod\,12\right)$
respectively for $s^{2}\equiv\left\{ \left(1\,\textnormal{or}\,9\right),\left(9\,\textnormal{or}\,1\right)\right\} \left(mod\,12\right)$,
yielding $m\equiv0\left(mod\,2\right)$ and $M\equiv2\left(mod\,24\right)$.
Furthermore, 

- if $m\equiv0\left(mod\,6\right)$, $a\equiv0$ or $2\left(mod\,3\right)$
and $s\equiv1$ or $5\left(mod\,6\right)$; 

- if $m\equiv2\left(mod\,6\right)$, $a\equiv1\left(mod\,3\right)$
and $s\equiv3\left(mod\,6\right)$; and 

- if $m\equiv4\left(mod\,6\right)$, $a$ can take any value and $s\equiv1\left(mod\,2\right)$.

(iv) for $\mu=4$, $\left(2\left(2a^{2}+1\right)+2m-s^{2}\right)\equiv0\left(mod\,12\right)$. 

For $\left(2\left(2a^{2}+1\right)\right)\equiv\left\{ 2,6\right\} \left(mod\,12\right)$,
$2m\equiv\left\{ \left(2\,\textnormal{or}\,10\right),\left(6\,\textnormal{or}\,10\right)\right\} \left(mod\,12\right)$
respectively for $s^{2}\equiv\left\{ \left(4\,\textnormal{or}\,0\right),\left(0\,\textnormal{or}\,4\right)\right\} \left(mod\,12\right)$,
yielding $m\equiv1\left(mod\,2\right)$ and $M\equiv16\left(mod\,24\right)$.
Furthermore, 

- if $m\equiv1\left(mod\,6\right)$, $a\equiv0\left(mod\,3\right)$
and $s\equiv2$ or $4\left(mod\,6\right)$; 

- if $m\equiv3\left(mod\,6\right)$, $a\equiv1$ or $2\left(mod\,3\right)$
and $s\equiv0\left(mod\,6\right)$; and 

- if $m\equiv5\left(mod\,6\right)$, $a$ can take any value and $s\equiv0\left(mod\,2\right)$.

(v) for $\mu=9$, $\left(9a^{2}+2m-s^{2}\right)\equiv0\left(mod\,12\right)$.

For $\left(9a^{2}\right)\equiv\left\{ 0,9\right\} \left(mod\,12\right)$,
$2m\equiv\left\{ \left(0\,\textnormal{or}\,4\right),\left(0\,\textnormal{or}\,4\right)\right\} \left(mod\,12\right)$
respectively for $s^{2}\equiv\left\{ \left(0\,\textnormal{or}\,4\right),\left(9\,\textnormal{or}\,1\right)\right\} \left(mod\,12\right)$,
yielding $m\equiv0$ or $2\left(mod\,6\right)$ and $M\equiv9$ or
$33\left(mod\,72\right)$. Furthermore,

- if $m\equiv0\left(mod\,6\right)$, either $a\equiv0\left(mod\,2\right)$
and $s\equiv0\left(mod\,6\right)$, or $a\equiv1\left(mod\,2\right)$
and $s\equiv3\left(mod\,6\right)$; and

- if $m\equiv2\left(mod\,6\right)$, either $a\equiv0\left(mod\,2\right)$
and $s\equiv2$ or $4\left(mod\,6\right)$, or $a\equiv1\left(mod\,2\right)$
and $s\equiv1$ or $5\left(mod\,6\right)$.

(vi) for $\mu=11$, $\left(11a^{2}+2a+1+2m-s^{2}\right)\equiv0\left(mod\,12\right)$. 

For $\left(11a^{2}+2a+1\right)\equiv\left\{ 1,2,5,10\right\} \left(mod\,12\right)$,
$2m\equiv\left\{ \left(0\,\textnormal{or}\,8\right),\left(2\,\textnormal{or}\,10\right),\left(4\,\textnormal{or}\,8\right),\right.$
$\left.\left(2\,\textnormal{or}\,6\right)\right\} \left(mod\,12\right)$
respectively for $s^{2}\equiv\left\{ \left(1\,\textnormal{or}\,9\right),\left(4\,\textnormal{or}\,0\right),\left(9\,\textnormal{or}\,1\right),\left(0\,\textnormal{or}\,4\right)\right\} \left(mod\,12\right)$
yielding 

- if $m\equiv0\left(mod\,6\right)$, $a\equiv0$ or $2\left(mod\,6\right)$
and $s\equiv1$ or $5\left(mod\,6\right)$; 

- if $m\equiv1\left(mod\,6\right)$, either $a\equiv1\left(mod\,6\right)$
and $s\equiv2$ or $4\left(mod\,6\right)$, or $a\equiv3$ or $5\left(mod\,6\right)$
and $s\equiv0\left(mod\,6\right)$; 

- if $m\equiv2\left(mod\,6\right)$, $a\equiv4\left(mod\,6\right)$
and $s\equiv3\left(mod\,6\right)$; 

- if $m\equiv3\left(mod\,6\right)$, $a\equiv3$ or $5\left(mod\,6\right)$
and $s\equiv2$ or $4\left(mod\,6\right)$; 

- if $m\equiv4\left(mod\,6\right)$, either $a\equiv0$ or $2\left(mod\,6\right)$
and $s\equiv3\left(mod\,6\right)$, or $a\equiv4\left(mod\,6\right)$
and $s\equiv1$ or $5\left(mod\,6\right)$; and 

- if $m\equiv5\left(mod\,6\right)$, $a\equiv1\left(mod\,6\right)$
and $s\equiv0\left(mod\,6\right)$.

Therefore, the congruences of Table 6 hold.
\end{proof}
Additional necessary conditions can be found using Beeckmans' necessary
conditions and are given in \cite{key-4}. Theorem 5 yields also that
$M$ can only be congruent to $0,1,2,9,11,16,23,24,25,26,33,35,40,47,49,50,59,64$
or $71\left(mod\,72\right)$.

The values of $M$ yielding solutions to (\ref{eq:53-3}) are given
in \cite{key-1-14}.

\section{Case of $M$ being square}

An interesting case occurs when $M$ is itself a squared integer as
shown in the following theorem.
\begin{thm}
For $M>1\in\mathbb{Z}^{+}$, $n\in\mathbb{Z}$, if $M$ is a square
integer, then there exist $M$ satisfying $M\equiv1\left(mod\,24\right)$
such as the sums of $M$ consecutive squared integers $\left(a+i\right)^{2}$
equal integer squares $s^{2}$; furthermore $M=\left(6n-1\right)^{2}$,
i.e $\left(M-1\right)/24$ are all generalized pentagonal numbers
$n\left(3n-1\right)/2$.\end{thm}
\begin{proof}
For $M>1,m,m_{1},m_{2}\in\mathbb{Z}^{+}$, $n\in\mathbb{Z}$, let
$M=m^{2}$; then $m\neq0\left(mod\,2\right)$ and $m\neq0\left(mod\,3\right)$
by (C1.1) and (C1.2). Therefore, $m\equiv\pm1\left(mod\,6\right)$
$\Rightarrow$ $m=6m_{1}\pm1$, yielding $M=12m_{1}\left(3m_{1}\pm1\right)+1$
or $M\equiv1\left(mod\,12\right)$. Then, by Theorem 5, $M\equiv1\left(mod\,24\right)$
$\Rightarrow$ $M=24m_{2}+1$, and $24m_{2}+1=\left(6m_{1}\pm1\right)^{2}$,
or $m_{2}=m_{1}\left(3m_{1}\pm1\right)/2$ which is equivalent to
$n\left(3n-1\right)/2$, $\forall n\in\mathbb{Z}$.
\end{proof}
The generalized pentagonal numbers $n\left(3n-1\right)/2$ \cite{key-8,key-6}
take the values 

$0,1,2,5,7,12,15,22,26,35,40,51,57,...$ \cite{key-5}, which then
yields 

$M=1,25,49,121,169,289,361,529,625,841,961,1225,1369,...$ The first
two values, $M=1,25$, should be rejected, the first one because $M$
must be greater than $1$, and the second one because for $M=25$,
one finds the unique solution $a=0$ and $s=70$ and $a$ must be
positive, although it is obviously equivalent to the solution with
$a=1$ and $s=70$ for $M=24$ of Lucas' cannonball problem (see also
\cite{key-7}).

\section{Conclusions}

It was shown that the problem of finding all the integer solutions
of the sum of $M$ consecutive integer squares starting at $a^{2}\geq1$
being equal to a squared integer $s^{2}$ has no solutions if $M$
is congruent to $3,5,6,7,8$ or $10\left(mod\,12\right)$ using Beeckmans
necessary conditions. It was further proven that the problem has integer
solutions if $M$ is congruent to $0,9,24$ or $33\left(mod\,72\right)$;
or to $1$, $2$ or $16\left(mod\,24\right)$; or to $11\left(mod\,12\right)$.
If $M$ is a square itself, then $M$ must be congruent to $1\left(mod\,24\right)$
and $\left(M-1\right)/24$ are all pentagonal numbers, except the
first two.

In a second paper \cite{key-7}, the Diophantine quadratic equation
(\ref{eq:53-3}) in variables $a$ and $s$ with $M$ as a parameter
is solved generally.

\section{Acknowledgment}

The author acknowledges Dr C. Thiel for the help brought throughout
this paper.

\end{document}